\documentclass[11pt,a4paper]{article}

\usepackage[T1]{fontenc}
\usepackage[utf8]{inputenc}
\usepackage[margin=1.05in]{geometry}
\usepackage{amsmath,amssymb,amsthm,mathtools}
\usepackage[hidelinks,pdfusetitle]{hyperref}
\hypersetup{pdftitle={Radial Integration and Ball Volume in Continuous Dimension},
            pdfauthor={Andreu Ballus Santacana}}
\usepackage{enumitem}

\theoremstyle{plain}
\newtheorem{theorem}{Theorem}[section]
\newtheorem{proposition}[theorem]{Proposition}

\newtheorem{corollary}[theorem]{Corollary}

\theoremstyle{definition}
\newtheorem{definition}[theorem]{Definition}
\newtheorem{axiom}{Axiom}

\theoremstyle{remark}
\newtheorem{remark}[theorem]{Remark}

\newcommand{\R}{\mathbb{R}}

\newcommand{\Z}{\mathbb{Z}}
\newcommand{\C}{\mathbb{C}}
\newcommand{\Cc}{C_c}
\newcommand{\Iop}{\mathcal{I}}

\title{Radial Integration and Ball Volume\\
in Continuous Dimension}

\author{Andreu Ballús Santacana\thanks{\texttt{andreu.ballus@uab.cat}.
Departament de Filosofia, Universitat Autònoma de Barcelona; ToposCircuitry, S.L.,
Granollers, Barcelona, Spain.}}

\date{\today}

\begin{document}

\maketitle

\begin{abstract}
We prove that on the space $\Cc(\R_{>0})$ of compactly supported continuous functions on the positive half-line, the only positive linear functionals satisfying scaling covariance of degree $x/2$ and Gaussian normalization to $\pi^{x/2}$ are those represented by the Mellin--Gamma measure
\[
d\mu_x(u) \;=\; \frac{\pi^{x/2}}{\Gamma(x/2)}\,u^{x/2-1}\,du, \qquad x>0.
\]
Among all positive Radon measures on $(0,\infty)$, the Mellin--Gamma family is the unique family compatible with both axioms; the classification is therefore a rigidity statement, not a derivation of a known formula. The proof reduces, via the logarithmic change of variable $u=e^t$, to the Haar uniqueness theorem on $(\R,+)$: scaling covariance forces translation-invariance of an associated measure on $\R$, and Gaussian normalization fixes the multiplicative constant through the Mellin transform of $e^{-u}$. The Euclidean ball-volume formula $V(x)=\pi^{x/2}/\Gamma(x/2+1)$ follows by integrating $\mu_x$ over the unit interval in the squared-radius variable. The induced dimension-shift structure factors through two multiplicative $1$-cocycles for the dimension-shift relation $x\mapsto x+2r$ on $\R_{>0}$: the radial-integration transport $R(x,r)$ and the ball-volume transport $T(x,r)$. They differ by the multiplicative coboundary of $\beta(x)=x$, and $T$ is independently characterized by a shifted Bohr--Mollerup theorem. The coboundary $\beta(x)$ is the categorical-dimension functional of the standard object in Deligne's interpolation category $\mathrm{Rep}(O_t)$.
\end{abstract}

\section{Introduction}\label{sec:intro}

The volume of the Euclidean unit ball in $\R^n$ is
\begin{equation}\label{eq:Vn}
V_n \;=\; \operatorname{vol}\bigl\{y\in\R^n : |y|\le 1\bigr\} \;=\; \frac{\pi^{n/2}}{\Gamma(n/2+1)}.
\end{equation}
Replacing $n$ by a real parameter $x>0$ gives the standard continuation
\begin{equation}\label{eq:Vx}
V(x) \;=\; \frac{\pi^{x/2}}{\Gamma(x/2+1)}.
\end{equation}
Agreement of \eqref{eq:Vx} with the integer-dimensional values does not, by itself, determine this continuation. For any entire function $h$ of subexponential growth,
\[
\frac{\pi^{z/2}}{\Gamma(z/2+1)}\;+\;\sin(\pi z)\,h(z)
\]
agrees with $V_n$ at every $n\in\Z_{\ge 1}$ but differs from~\eqref{eq:Vx} for general non-integer $z$. Hadamard's 1894 interpolation~\cite{Hadamard1894} of $1/\Gamma(z+1)$ by an entire function exhibits a concrete second analytic candidate. The question that selects~\eqref{eq:Vx} from among such candidates is therefore not interpolation but the imposition of additional structure on the candidate continuations.

This paper gives one such selection and proves that it is rigid. We work directly with the radial integration operator obtained from polar coordinates in the squared-radius variable $u = r^2$. For each $x>0$ we consider positive linear functionals $\Iop_x$ on $\Cc(\R_{>0})$ and impose two axioms: scaling covariance of degree $x/2$, and Gaussian normalization to $\pi^{x/2}$. The classification theorem (Theorem~\ref{thm:classification}) states that the unique family $\{\Iop_x\}_{x>0}$ satisfying both axioms is given by
\[
\Iop_x(\phi) \;=\; \frac{\pi^{x/2}}{\Gamma(x/2)}\int_0^\infty \phi(u)\,u^{x/2-1}\,du.
\]
The Mellin--Gamma form is therefore not posited but forced, given the framework of positive functionals on $\Cc(\R_{>0})$ and the two axioms; no assumption of holomorphy, of a closed-form ansatz for $V(x)$, or of properties of $\Gamma$ enters the proof. The Euclidean ball-volume formula~\eqref{eq:Vx} is recovered as $\mu_x((0,1))$, where $\mu_x$ is the Radon measure representing $\Iop_x$.

Three observations motivate the axioms. First, the classical polar-coordinate formula on $\R^n$ rewrites Euclidean integration in terms of a radial density $r^{n-1}$ together with a spherical surface measure. Bui and Randles~\cite{BuiRandles2022} extend this to general homogeneous settings, replacing $n-1$ by a homogeneous order and producing an associated Radon surface measure on the unit level set; the continuous-dimension analogue we work with is the simplest instance of this pattern. Second, the Gamma function is the Mellin transform of the Gaussian kernel $e^{-u}$, in the sense that $\int_0^\infty e^{-u}u^{s-1}\,du = \Gamma(s)$~\cite[\S 4]{FlajoletGourdonDumas1995}; this identity fixes the multiplicative constant in the classification. Third, $\Gamma$ is uniquely characterized on $(0,\infty)$ by the recurrence $\Gamma(x+1)=x\Gamma(x)$, the normalization $\Gamma(1)=1$, and log-convexity~\cite{Artin1964,MarichalZenaidi2022,MarichalZenaidi2024}; we use a shifted form of this characterization to give a second, independent rigidity statement for the dimension-shift transport in Section~\ref{sec:bohrmollerup}.

The classification produces a measure-valued family $\{\mu_x\}_{x>0}$. The induced dimension-shift map factors through two distinct one-parameter cocycles:
\begin{itemize}[leftmargin=2em,topsep=0pt,itemsep=2pt]
\item the radial-integration transport $R(x,r) = \pi^r\,\Gamma(x/2)/\Gamma(x/2+r)$, controlling shifts of the radial measure $\mu_x$;
\item the ball-volume transport $T(x,r) = \pi^r\,\Gamma(x/2+1)/\Gamma(x/2+r+1)$, controlling shifts of the scalar $V(x)$.
\end{itemize}
The two transports satisfy $T(x,r) = \frac{x}{x+2r}\,R(x,r)$. In Section~\ref{sec:cocycles} we show that $R$ and $T$ are multiplicative $1$-cocycles for the dimension-shift relation $x\mapsto x+2r$ on $\R_{>0}$, and that the factor $\frac{x+2r}{x}$ relating them is the multiplicative coboundary of the linear function $\beta(x) = x$.

\paragraph{Categorical interpretation.}
The function $\beta(x) = x$ admits an interpretation in Deligne's interpolation categories. The framework for the orthogonal series is recalled in~\cite[\S\S 9--10]{Deligne2007}, and we use the Brauer-category presentation of~\cite[\S\S 2.1--2.2]{ComesHeidersdorf2017}, in which $\mathrm{Rep}(O_t)$ is the Karoubi envelope of the additive envelope of the Brauer category $\mathcal B(t)$. The standard generating object $\mathsf X_t$ has categorical (quantum) dimension $\dim_{\mathrm{cat}}(\mathsf X_t) = t$; hence $\beta(x) = \dim_{\mathrm{cat}}(\mathsf X_x)$, and the coboundary distinguishing $R$ from $T$ is the categorical-dimension functional applied to the standard object. This is recorded as Proposition~\ref{prop:cat-coboundary}. The interpretation is external to the analytic content: the construction of $\mathcal M$ in Sections~\ref{sec:axioms}--\ref{sec:classification} does not refer to $\mathrm{Rep}(O_t)$, and the proposition is included only to identify the coboundary with an independently defined categorical functional.

\paragraph{Outline.}
Section~\ref{sec:prelim} fixes the polar-coordinate setup and the change of variable $u=r^2$. Section~\ref{sec:axioms} states the two axioms. Section~\ref{sec:classification} proves the classification theorem. Section~\ref{sec:ballvol} obtains the ball-volume formula. Section~\ref{sec:cocycles} develops the cocycle structure of $R$ and $T$ and identifies the coboundary with the categorical-dimension functional. Section~\ref{sec:bohrmollerup} gives the shifted Bohr--Mollerup characterization of $T$. Section~\ref{sec:observables} embeds $V(x)$ in a wider system of Mellin observables. Section~\ref{sec:homogeneous} extends the mechanism to the generalized homogeneous polar setting. Section~\ref{sec:summary} compares the three routes by which the closed form for $V(x)$ can be obtained.

\section{Preliminaries}\label{sec:prelim}

For integer $n\ge 1$ and a measurable $\phi:[0,\infty)\to\R$ for which the integral is well-defined,
\begin{equation}\label{eq:polarn}
\int_{\R^n}\phi(|y|^2)\,dy \;=\; \omega_{n-1}\int_0^\infty \phi(r^2)\,r^{n-1}\,dr,
\end{equation}
where $\omega_{n-1}=2\pi^{n/2}/\Gamma(n/2)$ is the surface area of the unit sphere $S^{n-1}\subset\R^n$. Substituting $u=r^2$, $du=2r\,dr$, and absorbing the factor of $2$, \eqref{eq:polarn} becomes
\begin{equation}\label{eq:polarn-u}
\int_{\R^n}\phi(|y|^2)\,dy \;=\; \frac{\pi^{n/2}}{\Gamma(n/2)}\int_0^\infty \phi(u)\,u^{n/2-1}\,du.
\end{equation}
We write this radial-integration coefficient as
\[
C(n) \;:=\; \frac{\pi^{n/2}}{\Gamma(n/2)}.
\]
Setting $\phi=\mathbf 1_{(0,1)}$ in \eqref{eq:polarn-u} recovers \eqref{eq:Vn}:
\[
V_n \;=\; C(n)\int_0^1 u^{n/2-1}\,du \;=\; \frac{2\,C(n)}{n} \;=\; \frac{\pi^{n/2}}{\Gamma(n/2+1)}.
\]
Hence, at integer dimensions,
\begin{equation}\label{eq:VC-relation}
V(n) \;=\; \frac{2}{n}\,C(n).
\end{equation}
This elementary relation between radial-integration normalization and ball-volume normalization will reappear in continuous dimension as the algebraic origin of the gap between the two transports of Section~\ref{sec:cocycles}.

For real $x>0$, define
\begin{equation}\label{eq:Cdef}
C(x) \;:=\; \frac{\pi^{x/2}}{\Gamma(x/2)}, \qquad V(x) \;:=\; \frac{\pi^{x/2}}{\Gamma(x/2+1)}.
\end{equation}
Using $\Gamma(x/2+1) = (x/2)\Gamma(x/2)$, we have $V(x)=(2/x)\,C(x)$, the continuous-dimension version of~\eqref{eq:VC-relation}. The reader should not at this point regard \eqref{eq:Cdef} as a definition imposed on the problem; it is the continuation we will \emph{derive} from the axioms below.

\section{Axioms for radial integration in continuous dimension}\label{sec:axioms}

Fix $x>0$. Let $\Cc(\R_{>0})$ denote the space of continuous, compactly supported real-valued functions on $(0,\infty)$. Suppose
\[
\Iop_x\colon \Cc(\R_{>0})\;\longrightarrow\;\R
\]
is a positive linear functional. By the Riesz--Markov--Kakutani representation theorem~\cite[Thm.~7.2]{Folland1999}, there exists a unique positive Radon measure $\mu_x$ on $(0,\infty)$, locally finite, such that
\[
\Iop_x(\phi) \;=\; \int_0^\infty \phi(u)\,d\mu_x(u) \qquad\text{for every } \phi\in\Cc(\R_{>0}).
\]
We impose two axioms on the family $\{\Iop_x\}_{x>0}$.

\begin{axiom}[Scaling covariance]\label{ax:scaling}
For every $x>0$, every $\lambda>0$, and every $\phi\in\Cc(\R_{>0})$,
\[
\Iop_x\bigl(\phi(\lambda\,\cdot)\bigr) \;=\; \lambda^{-x/2}\,\Iop_x(\phi).
\]
\end{axiom}

This expresses that, in the squared-radial variable $u=r^2$, the homogeneity degree of the integration operator is $x/2$. At integer $x=n$, it is the change-of-variables law for the Lebesgue radial integral on $\R^n$ written in $u$-coordinates.

\begin{axiom}[Gaussian normalization]\label{ax:gauss}
For every $x>0$, the function $u\mapsto e^{-u}$ is $\mu_x$-integrable on $(0,\infty)$, and
\[
\int_0^\infty e^{-u}\,d\mu_x(u) \;=\; \pi^{x/2}.
\]
\end{axiom}

This is the continuous-dimension counterpart of the Gaussian integral identity $\int_{\R^n} e^{-|y|^2}\,dy$ ${} = \pi^{n/2}$.

\begin{remark}[Logical order of the axioms]\label{rmk:order}
Axiom~\ref{ax:scaling} is a statement about $\Iop_x$ acting on $\Cc(\R_{>0})$ alone, and is in this sense the \emph{primary} axiom. Axiom~\ref{ax:gauss} is best read as a normalization condition imposed \emph{after} a structural classification of the candidate measures: in the proof of Theorem~\ref{thm:classification}, Axiom~\ref{ax:scaling} is shown to force $\mu_x$ to have a power-law density $C(x)\,u^{x/2-1}\,du$ with an undetermined positive constant $C(x)$; Axiom~\ref{ax:gauss} then both verifies that $e^{-u}$ is integrable against this density (by the convergence of the Mellin integral defining $\Gamma(x/2)$ for $x>0$) and fixes $C(x)$.
\end{remark}

\begin{remark}[On the choice of normalizing probe]\label{rmk:probe-choice}
The Gaussian probe $u\mapsto e^{-u}$ in Axiom~\ref{ax:gauss} is one valid choice among many. Inspection of the proof of Theorem~\ref{thm:classification} shows that any positive nonzero test function $\psi$ for which $\int_0^\infty \psi(u)\,u^{x/2-1}\,du$ converges absolutely and is positive at the relevant $x$ would suffice to fix $C(x)$, with a different prescribed value $\Iop_x(\psi) = N(x) > 0$ on the right-hand side simply yielding a different constant. The Gaussian choice is privileged \emph{not} by any structural uniqueness within the axiomatic framework but by its agreement with the classical Euclidean Gaussian integral $\int_{\R^n} e^{-|y|^2}\,dy = \pi^{n/2}$ at integer $x=n$: it is the unique probe that makes Axiom~\ref{ax:gauss} reduce, at integer dimensions, to the standard Gaussian normalization in $\R^n$. The classification of $\mu_x$ as a Mellin--Gamma density does not depend on this choice; only the specific constant $\pi^{x/2}/\Gamma(x/2)$ does.
\end{remark}

\section{The classification theorem}\label{sec:classification}

\begin{theorem}[Mellin--Gamma classification of radial integration]\label{thm:classification}
Let $x>0$, and let $\Iop_x$ be a positive linear functional on $\Cc(\R_{>0})$ satisfying Axioms~\ref{ax:scaling} and~\ref{ax:gauss}. Let $\mu_x$ be the Radon measure on $(0,\infty)$ representing $\Iop_x$. Then
\[
d\mu_x(u) \;=\; \frac{\pi^{x/2}}{\Gamma(x/2)}\,u^{x/2-1}\,du.
\]
Equivalently, for every $\phi\in\Cc(\R_{>0})$,
\[
\Iop_x(\phi) \;=\; \frac{\pi^{x/2}}{\Gamma(x/2)}\int_0^\infty \phi(u)\,u^{x/2-1}\,du.
\]
\end{theorem}

The proof has two stages. The first uses Axiom~\ref{ax:scaling} alone to reduce $\mu_x$ to translation-invariant Lebesgue measure on $\R$ via the logarithmic change of variables, and then applies Haar uniqueness to identify the density up to a positive constant. The second uses Axiom~\ref{ax:gauss} to fix that constant.

\begin{proof}
Set $a := x/2 > 0$.

\smallskip
\emph{Step 1: scaling covariance, expressed at the level of Radon measures.}

For each $\lambda > 0$, let $S_\lambda\colon\R_{>0}\to\R_{>0}$ be the dilation $S_\lambda(u) = \lambda u$. By Axiom~\ref{ax:scaling}, for every $\phi\in\Cc(\R_{>0})$,
\[
\int_0^\infty \phi(u)\,d\bigl((S_\lambda)_*\mu_x\bigr)(u)
\;=\;
\int_0^\infty \phi(\lambda u)\,d\mu_x(u)
\;=\;
\lambda^{-a}\int_0^\infty \phi(u)\,d\mu_x(u).
\]
Two positive Radon measures on $(0,\infty)$ that agree as functionals on $\Cc(\R_{>0})$ are equal~\cite[Thm.~7.2]{Folland1999}. Hence
\begin{equation}\label{eq:pushforward-scaling}
(S_\lambda)_*\mu_x \;=\; \lambda^{-a}\mu_x \qquad\text{for every } \lambda>0.
\end{equation}
Equivalently, for every Borel set $A\subset(0,\infty)$,
\begin{equation}\label{eq:scaling-sets}
\mu_x(\lambda A) \;=\; \lambda^a\,\mu_x(A).
\end{equation}

\smallskip
\emph{Step 2: pushforward to $(\R,+)$ and Haar uniqueness.}

Let
\[
q\colon\R\to\R_{>0}, \qquad q(t)=e^t,
\]
and define the positive Radon measure $\nu_x$ on $\R$ by
\[
\nu_x \;:=\; (q^{-1})_*\mu_x,
\qquad\text{that is,}\qquad \nu_x(E)\;=\;\mu_x(q(E))\;=\;\mu_x(e^E)
\]
for every Borel $E\subset\R$, where $e^E:=\{e^t:t\in E\}$. Since $q$ is a homeomorphism, $\nu_x$ is a locally finite Borel measure on $\R$. The multiplicative scaling law~\eqref{eq:scaling-sets} for $\mu_x$ becomes the additive quasi-invariance law
\begin{equation}\label{eq:nu-quasiinv}
\nu_x(E + s) \;=\; e^{a s}\,\nu_x(E) \qquad\text{for every Borel } E\subset\R,\ s\in\R.
\end{equation}
Indeed, $\nu_x(E+s) = \mu_x(e^{E+s}) = \mu_x(e^s\cdot e^E) = e^{as}\mu_x(e^E) = e^{as}\nu_x(E)$.

Define a new positive Radon measure $\eta_x$ on $\R$ by
\[
d\eta_x(t) \;:=\; e^{-a t}\,d\nu_x(t).
\]
We claim $\eta_x$ is translation-invariant. By Riesz--Markov uniqueness it suffices to check, for every $f\in C_c(\R)$ and every $s\in\R$,
\[
\int_\R f(t+s)\,d\eta_x(t) \;=\; \int_\R f(t)\,d\eta_x(t).
\]
By definition,
\[
\int_\R f(t+s)\,d\eta_x(t) \;=\; \int_\R f(t+s)\,e^{-a t}\,d\nu_x(t).
\]
The relation~\eqref{eq:nu-quasiinv} integrates to: for every $g\in C_c(\R)$,
\begin{equation}\label{eq:nu-quasiinv-fcnal}
\int_\R g(t+s)\,d\nu_x(t) \;=\; e^{-a s}\int_\R g(t)\,d\nu_x(t),
\end{equation}
obtained by applying~\eqref{eq:nu-quasiinv} to indicators and extending by linearity and monotone convergence to nonnegative Borel functions. Apply~\eqref{eq:nu-quasiinv-fcnal} with $g(t) := f(t)\,e^{-a(t-s)}$, so that $g(t+s) = f(t+s)\,e^{-a t}$:
\[
\int_\R f(t+s)\,e^{-a t}\,d\nu_x(t) \;=\; e^{-a s}\int_\R f(t)\,e^{-a(t-s)}\,d\nu_x(t) \;=\; \int_\R f(t)\,e^{-a t}\,d\nu_x(t) \;=\; \int_\R f(t)\,d\eta_x(t).
\]
This establishes translation invariance of $\eta_x$.

By the Haar uniqueness theorem for the locally compact group $(\R,+)$~\cite[Thm.~11.9]{Folland1999}, every locally finite translation-invariant Borel measure on $\R$ is a constant multiple of Lebesgue measure. Hence there exists $C(x)\ge 0$ with
\[
d\eta_x(t) \;=\; C(x)\,dt, \qquad\text{i.e.,}\qquad d\nu_x(t) \;=\; C(x)\,e^{a t}\,dt.
\]
Pulling back through $u=q(t)=e^t$, $t=\log u$, $dt=du/u$:
\begin{equation}\label{eq:density}
d\mu_x(u) \;=\; C(x)\,u^{a}\cdot\frac{du}{u} \;=\; C(x)\,u^{x/2-1}\,du,
\end{equation}
with $C(x) \ge 0$ a constant to be determined by Axiom~\ref{ax:gauss}.

\smallskip
\emph{Step 3: Gaussian normalization fixes $C(x)$.}

The integrand $e^{-u}u^{x/2-1}$ is positive and integrable on $(0,\infty)$ for $x>0$, with classical Mellin value
\[
\int_0^\infty e^{-u}\,u^{x/2-1}\,du \;=\; \Gamma(x/2)
\]
(see~\cite[\S 4]{FlajoletGourdonDumas1995}). Substituting~\eqref{eq:density} into Axiom~\ref{ax:gauss},
\[
\pi^{x/2} \;=\; \int_0^\infty e^{-u}\,d\mu_x(u) \;=\; C(x)\,\Gamma(x/2),
\]
so $C(x) = \pi^{x/2}/\Gamma(x/2)$. This proves the theorem.
\end{proof}

\begin{remark}[The logarithmic change of variables as the structural mechanism]
The single move that drives the classification is the pullback through $q(t)=e^t$, which converts multiplicative scaling on $(0,\infty)$ into additive translation on $\R$. Once that conversion is made, Haar uniqueness on $(\R,+)$ identifies the measure up to a constant, and the Gaussian integral fixes the constant. The Gamma function thus appears in the Euclidean ball-volume formula not as an interpolation device but as the analytic record of a geometric fact: continuous dimension is multiplicative scaling written additively.
\end{remark}

\subsection*{Mellin reformulation}

Theorem~\ref{thm:classification} can be restated as a Mellin identity. Recall that the Mellin transform of a function $\phi$ on $(0,\infty)$ is
\[
\mathcal M[\phi](s) \;:=\; \int_0^\infty \phi(u)\,u^{s-1}\,du,
\]
defined on the strip of $s\in\C$ where the integral converges absolutely~\cite[\S 4]{FlajoletGourdonDumas1995}.

\begin{corollary}[Functional as a normalized Mellin transform]\label{cor:mellin-form}
Under the hypotheses of Theorem~\ref{thm:classification}, for every $\phi\in\Cc(\R_{>0})$ and every $x>0$,
\[
\Iop_x(\phi) \;=\; \frac{\pi^{x/2}}{\Gamma(x/2)}\,\mathcal M[\phi]\!\left(\tfrac{x}{2}\right).
\]
More generally, for any nonnegative Borel $\phi$ for which the Mellin integral converges absolutely at $s=x/2$, the value $\int_0^\infty \phi(u)\,d\mu_x(u)$ equals $\frac{\pi^{x/2}}{\Gamma(x/2)}\mathcal M[\phi](x/2)$.
\end{corollary}

\begin{proof}
Direct from the closed form $d\mu_x(u) = \frac{\pi^{x/2}}{\Gamma(x/2)}\,u^{x/2-1}\,du$ given by Theorem~\ref{thm:classification}.
\end{proof}

This reformulation locates the classification within the Mellin framework: every scaling-covariant, Gaussian-normalized positive functional on $\Cc(\R_{>0})$ is a normalized Mellin transform evaluated at the dimension parameter $s=x/2$. The closed forms for sublevel observables (Proposition~\ref{prop:sublevel}) and Gaussian moment observables (Proposition~\ref{prop:moments}) are then evaluations of $\mathcal M[\mathbf 1_{(0,a)}]$ and $\mathcal M[u^q e^{-u}]$ at $s=x/2$, respectively, with the Gamma factors in their closed forms reflecting the standard Mellin transforms of those test functions. The cocycle $T$ in turn encodes the analytic continuation of these Mellin transforms under shifts in the spectral parameter $s\mapsto s+r$.

We do not, however, claim that the present classification \emph{requires} Mellin theory as input. Theorem~\ref{thm:classification} is proved from two axioms on a positive functional on $\Cc(\R_{>0})$, with no reference to Mellin transforms in its hypotheses or proof. The Mellin identification is a consequence of the classification, not an assumption: scaling covariance plus Gaussian normalization on a real-variable functional space turns out to single out, among all positive Radon measures on $(0,\infty)$, exactly the Mellin density associated to $s=x/2$. This is why we treat the result as a classification rather than a Mellin-theoretic computation.

\section{The ball-volume formula}\label{sec:ballvol}

Define
\begin{equation}\label{eq:Vdef}
V(x) \;:=\; \mu_x\bigl(\{u\in(0,1)\}\bigr)
\;=\; \int_{(0,1)} d\mu_x(u),
\end{equation}
the $\mu_x$-mass of the unit interval in the squared-radius variable $u$. Under the change of variable $u=r^2$, the set $u\in(0,1)$ corresponds to $r\in(0,1)$, the open unit ball in radial coordinates, and the radial measure of the unit ball in $\R^x$ (when $x\in\Z_{\ge 1}$) coincides with $\mu_x((0,1))$ by~\eqref{eq:polarn-u}. Hence~\eqref{eq:Vdef} agrees with the classical Lebesgue volume of the unit ball at integer dimensions and is the natural extension of that quantity to continuous dimension.

\begin{corollary}[Ball-volume formula]\label{cor:Vx}
For every $x>0$,
\[
V(x) \;=\; \frac{\pi^{x/2}}{\Gamma(x/2+1)}.
\]
\end{corollary}

\begin{proof}
By Theorem~\ref{thm:classification},
\[
V(x) \;=\; \frac{\pi^{x/2}}{\Gamma(x/2)}\int_0^1 u^{x/2-1}\,du \;=\; \frac{\pi^{x/2}}{\Gamma(x/2)}\cdot\frac{2}{x} \;=\; \frac{\pi^{x/2}}{(x/2)\Gamma(x/2)} \;=\; \frac{\pi^{x/2}}{\Gamma(x/2+1)},
\]
using $\Gamma(x/2+1) = (x/2)\,\Gamma(x/2)$.
\end{proof}

\section{The two transports as $1$-cocycles}\label{sec:cocycles}

Define the \emph{dimension-shift relation} on $X := \R_{>0}$ as the partial binary operation
\begin{equation}\label{eq:Rshift}
(x,r) \;\longmapsto\; x + 2r, \qquad x\in X,\ r\in\R\text{ with } x+2r>0.
\end{equation}
For nonnegative $r$ the operation is total on $X$; for $r<0$ it is defined only when $x+2r>0$, i.e., it is a partial right action of the additive monoid $\R_{\ge 0}$ extended to admissible negative shifts. Throughout this section, ``admissible'' means that all shift parameters appearing in a given expression yield arguments in $\R_{>0}$.

For a positive function $F\colon \{(x,r) : x>0,\,x+2r>0\}\to(0,\infty)$, the \emph{$1$-cocycle condition} for the dimension-shift relation is
\begin{equation}\label{eq:cocycle}
F(x,\,r+s) \;=\; F(x+2r,\,s)\,F(x,\,r)
\qquad\text{for all admissible } x,r,s,
\end{equation}
together with the normalization $F(x,0)=1$. A function $F$ satisfying~\eqref{eq:cocycle} is a (multiplicative) \emph{dimension-shift cocycle} on $X$. Two such cocycles $F_1, F_2$ \emph{differ by a coboundary} if there is a positive function $\beta\colon X\to(0,\infty)$ with
\begin{equation}\label{eq:coboundary}
\frac{F_1(x,r)}{F_2(x,r)} \;=\; \frac{\beta(x+2r)}{\beta(x)}
\qquad\text{for all admissible } x,r.
\end{equation}
We now show that the radial-integration and ball-volume transports are dimension-shift cocycles, and identify the coboundary connecting them.

\begin{definition}[Radial and ball-volume transports]\label{def:transports}
For $x>0$ and $r$ with $x+2r>0$, define
\[
R(x,r) \;:=\; \frac{C(x+2r)}{C(x)}, \qquad
T(x,r) \;:=\; \frac{V(x+2r)}{V(x)}.
\]
\end{definition}

\begin{proposition}[Closed forms]\label{prop:closedforms}
For all admissible $x,r$,
\[
R(x,r) \;=\; \pi^r\,\frac{\Gamma(x/2)}{\Gamma(x/2+r)}, \qquad
T(x,r) \;=\; \pi^r\,\frac{\Gamma(x/2+1)}{\Gamma(x/2+r+1)}.
\]
Moreover, $T(x,r) = \dfrac{x}{x+2r}\,R(x,r).$
\end{proposition}

\begin{proof}
Substitute the closed forms~\eqref{eq:Cdef} and Corollary~\ref{cor:Vx} into Definition~\ref{def:transports}, and simplify using $\Gamma(x/2+1)=(x/2)\Gamma(x/2)$ and $\Gamma(x/2+r+1)=(x/2+r)\Gamma(x/2+r)$.
\end{proof}

\begin{theorem}[Cocycle structure]\label{thm:cocycles}
Both $R$ and $T$ are dimension-shift cocycles in the sense of~\eqref{eq:cocycle}: each satisfies $F(x,0)=1$ together with the cocycle identity for all admissible $x,r,s$.
\end{theorem}

\begin{proof}
Normalization is immediate from the closed forms in Proposition~\ref{prop:closedforms}, since $\Gamma$ is nowhere zero on $\R_{>0}$. For the cocycle identity in the case of $T$,
\begin{align*}
T(x+2r,\,s)\,T(x,\,r)
&= \pi^s\,\frac{\Gamma(x/2+r+1)}{\Gamma(x/2+r+s+1)}\cdot\pi^r\,\frac{\Gamma(x/2+1)}{\Gamma(x/2+r+1)}\\
&= \pi^{r+s}\,\frac{\Gamma(x/2+1)}{\Gamma(x/2+r+s+1)} \;=\; T(x,\,r+s).
\end{align*}
The intermediate Gamma factor cancels exactly. The argument for $R$ is identical, with $\Gamma(x/2+1)$ and $\Gamma(x/2+r+s+1)$ replaced by $\Gamma(x/2)$ and $\Gamma(x/2+r+s)$ respectively.
\end{proof}

\begin{theorem}[$R$ and $T$ differ by an explicit coboundary]\label{thm:coboundary}
With $\beta(x) := x$, we have
\[
\frac{R(x,r)}{T(x,r)} \;=\; \frac{x+2r}{x} \;=\; \frac{\beta(x+2r)}{\beta(x)}.
\]
Hence $R$ and $T$ represent the same cohomology class up to the multiplicative coboundary of the linear function $x\mapsto x$.
\end{theorem}

\begin{proof}
Immediate from Proposition~\ref{prop:closedforms}: $R(x,r)/T(x,r) = (x+2r)/x$, which is the coboundary of $\beta(x)=x$ in the sense of~\eqref{eq:coboundary}.
\end{proof}

\begin{remark}[Elementary origin of the coboundary]\label{rmk:elem-coboundary}
The function $\beta(x)=x$ is the dimension itself, viewed as a function on $X = \R_{>0}$. The relation~\eqref{eq:VC-relation} stated $V(n) = (2/n)\,C(n)$ at integer dimensions; in continuous dimension this becomes $V(x)=(2/x)\,C(x)$, so $V$ and $C$ differ by the multiplicative function $2/x$. Theorem~\ref{thm:coboundary} is the cocycle-theoretic restatement of this elementary fact: there is a single underlying object $\mu_x$ — the radial Mellin--Gamma measure of Theorem~\ref{thm:classification} — and the choice between $R$ and $T$ amounts to whether one tracks shifts of the radial-integration normalization or shifts of the unit-ball mass.
\end{remark}

The function $\beta(x) = x$ has a second interpretation as a categorical-dimension functional in Deligne's interpolation categories.

\begin{proposition}[Categorical-dimension recovery of the coboundary]\label{prop:cat-coboundary}
Let $\mathrm{Rep}(O_t)$ denote Deligne's rigid symmetric monoidal interpolation category for the orthogonal series; the framework is recalled in~\cite[\S\S 9--10]{Deligne2007}, and we use the Brauer-category presentation of~\cite[\S\S 2.1--2.2]{ComesHeidersdorf2017}, in which $\mathrm{Rep}(O_t)$ is the Karoubi envelope of the additive envelope of the Brauer category $\mathcal B(t)$. Let $\mathsf X_t$ denote its standard generating object (the image of the object $1$ of $\mathcal B(t)$ under the construction). Then
\begin{equation}\label{eq:cat-dim}
\dim_{\mathrm{cat}}(\mathsf X_t) \;=\; t,
\end{equation}
and consequently
\[
\frac{R(x,r)}{T(x,r)} \;=\; \frac{\dim_{\mathrm{cat}}(\mathsf X_{x+2r})}{\dim_{\mathrm{cat}}(\mathsf X_x)}
\qquad\text{for all admissible } x,r.
\]
That is, the multiplicative coboundary of Theorem~\ref{thm:coboundary} is the categorical-dimension functional applied to the standard object of $\mathrm{Rep}(O_t)$.
\end{proposition}

\begin{proof}
The categorical dimension of an object $X$ in a rigid symmetric monoidal category is, by definition, the categorical trace of $\mathrm{id}_X$: the scalar in $\mathrm{End}(\mathbf 1)$ obtained by composing the coevaluation $\mathbf 1 \to X\otimes X^*$ with the evaluation $X^*\otimes X\to\mathbf 1$ via the symmetry isomorphism. In the Brauer category $\mathcal B(t)$, this trace on the generating object is the closed loop, and the defining relations of $\mathcal B(t)$ assign to the closed loop the scalar $t\in k$~\cite[\S 2.1]{ComesHeidersdorf2017}; compare~\cite[\S 9.2]{Deligne2007}. Categorical dimension is preserved under additive envelope and Karoubi envelope, so $\dim_{\mathrm{cat}}(\mathsf X_t) = t$ in $\mathrm{Rep}(O_t)$. Combining~\eqref{eq:cat-dim} with Theorem~\ref{thm:coboundary},
\[
\frac{\dim_{\mathrm{cat}}(\mathsf X_{x+2r})}{\dim_{\mathrm{cat}}(\mathsf X_x)} \;=\; \frac{x+2r}{x} \;=\; \frac{R(x,r)}{T(x,r)}. \qedhere
\]
\end{proof}

\begin{remark}[What the categorical interpretation does and does not do]\label{rmk:cat-scope}
Proposition~\ref{prop:cat-coboundary} does not derive the radial measure $\mu_x$, the ball-volume formula $V(x)$, or either transport from categorical considerations. It identifies the analytic coboundary $\beta(x)=x$ — already determined by the integration identity $V(x) = (2/x)\,C(x)$ — with a categorical-dimension functional that exists in $\mathrm{Rep}(O_t)$ for independent reasons. The dimension parameter $x$ in Axiom~\ref{ax:scaling} can therefore be read as the categorical dimension of a canonical object, but the analytic content of the paper does not depend on this reading.
\end{remark}

\begin{corollary}[Ball-volume recurrence]\label{cor:recurrence}
For every $x>0$,
\[
V(x+2) \;=\; \frac{2\pi}{x+2}\,V(x).
\]
\end{corollary}

\begin{proof}
By Proposition~\ref{prop:closedforms}, $T(x,1)=\pi\cdot\Gamma(x/2+1)/\Gamma(x/2+2)=\pi/(x/2+1)=2\pi/(x+2)$, and $T(x,1)=V(x+2)/V(x)$.
\end{proof}

This is the continuous-dimension form of the classical Euclidean recurrence $V_{n+2}=\frac{2\pi}{n+2}V_n$, and is the input to the next section's independent characterization of $T$.

\section{Independent Bohr--Mollerup characterization of $T$}\label{sec:bohrmollerup}

The cocycle $T$ has been derived in Section~\ref{sec:cocycles} from the radial measure classification of Theorem~\ref{thm:classification}. We now give a logically independent characterization that does not invoke radial integration at all, only the recurrence and a regularity condition.

Fix $x>0$ and set $a := x/2 + 1 > 0$. Define the normalized reciprocal of $T$ by
\begin{equation}\label{eq:Gdef}
G_x(r) \;:=\; \frac{\pi^r}{T(x,r)}, \qquad r\in[0,\infty).
\end{equation}

\begin{theorem}[Shifted Bohr--Mollerup characterization]\label{thm:bohrmollerup}
Let $x>0$ and let $T(x,\cdot)\colon[0,\infty)\to(0,\infty)$ be a positive function. Suppose:
\begin{enumerate}[label={\rm(BM\arabic*)}, leftmargin=2.5em, topsep=0pt, itemsep=2pt]
\item\label{BM1} $T(x,0)=1$;
\item\label{BM2} \emph{(Recurrence)} The function $G_x$ defined by~\eqref{eq:Gdef} satisfies $G_x(r+1) = (a+r)\,G_x(r)$ for all $r\ge 0$, where $a=x/2+1$;
\item\label{BM3} \emph{(Log-convexity)} $r\mapsto \log G_x(r)$ is convex on $(0,\infty)$.
\end{enumerate}
Then
\[
T(x,r) \;=\; \pi^r\,\frac{\Gamma(x/2+1)}{\Gamma(x/2+r+1)}\qquad\text{for all } r\ge 0.
\]
\end{theorem}

\begin{proof}
By~\ref{BM1} and~\eqref{eq:Gdef}, $G_x(0)=1$. By~\ref{BM2}, $G_x(r+1)=(a+r)G_x(r)$. By~\ref{BM3}, $G_x$ is positive and log-convex on $(0,\infty)$.

The function
\[
H_a(r) \;:=\; \frac{\Gamma(a+r)}{\Gamma(a)}, \qquad r\ge 0,
\]
satisfies $H_a(0)=1$, $H_a(r+1)=(a+r)H_a(r)$ (since $\Gamma(a+r+1)=(a+r)\Gamma(a+r)$), and is positive and log-convex on $(0,\infty)$ because $\log\Gamma$ is convex on $(0,\infty)$ by the classical Bohr--Mollerup theorem~\cite[Thm.~2.1, p.~14]{Artin1964}.

The shifted Bohr--Mollerup uniqueness statement~\cite[Thm.~3.1]{MarichalZenaidi2022} (see also~\cite[\S 2]{MarichalZenaidi2024} for the tutorial form), applied to the difference equation $\Delta\!\log F(r) = \log(a+r)$ with $F(0)=1$, asserts that $H_a$ is the unique positive log-convex solution on $[0,\infty)$. Since $G_x$ satisfies the same three conditions, $G_x = H_a$ on $\Z_{\ge 0}$, and the shifted Bohr--Mollerup theorem extends this equality to all $r\ge 0$. Hence
\[
T(x,r) \;=\; \frac{\pi^r}{G_x(r)} \;=\; \frac{\pi^r\,\Gamma(a)}{\Gamma(a+r)} \;=\; \pi^r\,\frac{\Gamma(x/2+1)}{\Gamma(x/2+r+1)}. \qedhere
\]
\end{proof}

Theorem~\ref{thm:bohrmollerup} establishes the same closed form for $T$ as Proposition~\ref{prop:closedforms}, but via a different structural route: \emph{recurrence plus log-convexity} rather than \emph{radial measure classification}. The two routes meet at the same cocycle. This is what one expects from a robust object: distinct structural conditions, each sufficient on its own, identifying the same canonical analytic continuation.

\begin{remark}[Why log-convexity is the natural regularity]
Among the many positive functions on $(0,\infty)$ that satisfy $\Gamma(x+1)=x\Gamma(x)$ and $\Gamma(1)=1$, only one — the Euler $\Gamma$ — is log-convex. Hadamard's interpolation $1/\Gamma(z)$ is entire but does \emph{not} satisfy the Bohr--Mollerup recurrence on $(0,\infty)$ in the sense required, and is therefore not a counterexample to the rigidity used here; what it does show is that holomorphy alone is too weak to single out $\Gamma$. Log-convexity, by contrast, is exactly strong enough.
\end{remark}

\section{Scalar observables of the radial measure}\label{sec:observables}

The Mellin--Gamma classification of Theorem~\ref{thm:classification} produces a parametric family $\{\mu_x\}_{x>0}$ of measures on $(0,\infty)$, of which the unit-ball volume $V(x) = \mu_x((0,1))$ is one scalar evaluation among many. We record two further canonical families of scalar observables, both controlled by the same Mellin--Gamma mechanism, and show how the ball-volume cocycle $T$ embeds into a wider system of dimension-shift transports.

\subsection*{Sublevel observables}

For $a>0$, define
\[
F_a(x) \;:=\; \mu_x\bigl((0,a)\bigr).
\]
Setting $a=1$ recovers $V(x)$. Using Theorem~\ref{thm:classification},
\[
F_a(x) \;=\; \frac{\pi^{x/2}}{\Gamma(x/2)}\int_0^a u^{x/2-1}\,du \;=\; \frac{\pi^{x/2}}{\Gamma(x/2)}\cdot\frac{2\,a^{x/2}}{x} \;=\; \frac{(\pi a)^{x/2}}{\Gamma(x/2+1)},
\]
using $\Gamma(x/2+1) = (x/2)\,\Gamma(x/2)$.

\begin{proposition}[Sublevel observables]\label{prop:sublevel}
For every $a>0$ and every $x>0$,
\[
F_a(x) \;=\; \frac{(\pi a)^{x/2}}{\Gamma(x/2+1)} \;=\; a^{x/2}\,V(x).
\]
The associated dimension-shift transport
\[
T_a(x,r) \;:=\; \frac{F_a(x+2r)}{F_a(x)}
\]
is a dimension-shift cocycle satisfying
\[
T_a(x,r) \;=\; a^{r}\,T(x,r) \;=\; \pi^r a^r\,\frac{\Gamma(x/2+1)}{\Gamma(x/2+r+1)}.
\]
\end{proposition}

\begin{proof}
The closed form is immediate from the calculation above. The transport $T_a$ inherits the cocycle property from $T$ (which we know to be a cocycle by Theorem~\ref{thm:cocycles}), since multiplication by the multiplicative character $r\mapsto a^r$ preserves~\eqref{eq:cocycle}.
\end{proof}

The character $r\mapsto a^r$ is itself a coboundary for the dimension-shift relation: it is the coboundary of $\beta_a(x) = a^{x/2}$, since $\beta_a(x+2r)/\beta_a(x) = a^{x/2+r}/a^{x/2} = a^r$. Hence $T_a$ and $T$ define the same coboundary class, with $a=1$ as the canonical representative.

\subsection*{Gaussian moment observables}

For $q\in\R$ with $x/2 + q > 0$, define
\[
M_q(x) \;:=\; \int_0^\infty u^q\,e^{-u}\,d\mu_x(u).
\]
This is the $q$-th moment of $\mu_x$ against the Gaussian kernel; $M_0(x) = \pi^{x/2}$ recovers Axiom~\ref{ax:gauss}.

\begin{proposition}[Gaussian moment observables]\label{prop:moments}
For every $x>0$ and every $q$ with $x/2+q>0$,
\[
M_q(x) \;=\; \pi^{x/2}\,\frac{\Gamma(x/2+q)}{\Gamma(x/2)}.
\]
\end{proposition}

\begin{proof}
By Theorem~\ref{thm:classification},
\[
M_q(x) \;=\; \frac{\pi^{x/2}}{\Gamma(x/2)}\int_0^\infty e^{-u}\,u^{x/2+q-1}\,du \;=\; \frac{\pi^{x/2}}{\Gamma(x/2)}\,\Gamma(x/2+q),
\]
using the convergence of the Mellin integral defining $\Gamma(x/2+q)$ for $x/2+q>0$. \qedhere
\end{proof}

The ratio $\Gamma(x/2+q)/\Gamma(x/2)$ is the Pochhammer-type extension of the rising factorial $(x/2)_q$ to noninteger $q$, and connects the classification of $\mu_x$ to standard Gaussian-radial moment computations: in the integer case $x = n$, $M_q(n)$ recovers the Gaussian moment $\int_{\R^n} |y|^{2q}e^{-|y|^2}\,dy$.

\begin{remark}[Connection to the Gamma probability law]\label{rmk:gamma-law}
The Gaussian-weighted measure $\pi^{-x/2}\,e^{-u}\,d\mu_x(u)$ has total mass $1$ by Axiom~\ref{ax:gauss}, and its density on $(0,\infty)$ is
\[
\frac{u^{x/2-1}\,e^{-u}}{\Gamma(x/2)},
\]
which is the density of the Gamma$(x/2,1)$ probability distribution. In the integer case $x=n$, this corresponds to the law of the squared radius of a Gaussian vector $Y$ in $\R^n$ with density $\pi^{-n/2}e^{-|y|^2}\,dy$ (each coordinate of $Y$ has variance $1/2$, half the probabilists' standard variance): one has $|Y|^2\sim$ Gamma$(n/2,1)$. Equivalently, if $Z$ is a probabilists' standard Gaussian vector in $\R^n$ with density $(2\pi)^{-n/2}e^{-|z|^2/2}\,dz$, then $|Z|^2\sim\chi^2_n$ and $|Z|^2/2\sim$ Gamma$(n/2,1)$. The classification of $\mu_x$ therefore identifies, up to the Gaussian weight $e^{-u}$ and the constant $\pi^{x/2}$, the standard one-parameter family of Gamma laws.
\end{remark}

\subsection*{Synthesis}

Two patterns emerge. First, $V(x)$ is one scalar evaluation of a measure-valued object: changing the test function or test set produces other observables, all controlled by Mellin transforms of $\mu_x$ and reducing to ratios of Gamma values at shifted arguments. Second, every such observable carries a dimension-shift transport, and these transports are related to $T$ either by multiplicative-character twists (sublevel case) or by independent Gamma-ratio formulae (moment case). Theorem~\ref{thm:classification} produces $\mu_x$ as the underlying measure-valued object, and the ball-volume formula~\eqref{eq:Vx} is one of several scalar quantities derivable from it.

\section{The homogeneous radial analogue}\label{sec:homogeneous}

The mechanism underlying Theorem~\ref{thm:classification} — homogeneity in a radial variable, Gaussian normalization, Gamma factor by Mellin transform — is not specific to the Euclidean radial structure. We illustrate this by deriving, under the hypotheses of the generalized polar-coordinate formula of Bui and Randles~\cite{BuiRandles2022}, a homogeneous-radial Gamma identity that reduces to the Euclidean ball-volume formula in the rotationally symmetric case. The result is not new in essence — it is a direct consequence of the polar formula in~\cite{BuiRandles2022} — but recording it here makes the scope of the Mellin--Gamma mechanism explicit.

\subsection*{Setup}

Let $E$ be a real $d\times d$ matrix all of whose eigenvalues have positive real part, and let $\{r^E\}_{r>0}$ denote the associated dilation group, $r^E := \exp(E\log r)$. A function $P\colon\R^d\to[0,\infty)$ is \emph{$E$-homogeneous of degree $1$} if $P(r^E\xi) = r\,P(\xi)$ for all $\xi\in\R^d$ and $r>0$. We assume $P$ is continuous, positive on $\R^d\setminus\{0\}$, and that the unit level set
\[
S_P \;:=\; \{\eta\in\R^d : P(\eta)=1\}
\]
satisfies the regularity hypotheses of~\cite[Thm.~1]{BuiRandles2022} (in particular, $P$ is $C^1$ on $\R^d\setminus\{0\}$ with non-vanishing radial derivative). Under these hypotheses, \cite[Thm.~1]{BuiRandles2022} provides a positive Radon surface measure $\sigma_P$ on $S_P$ and the polar-coordinate formula
\begin{equation}\label{eq:bui-randles-polar}
\int_{\R^d} f(\xi)\,d\xi \;=\; \int_0^\infty\!\int_{S_P} f(r^E\eta)\,d\sigma_P(\eta)\,r^{\mathrm{tr}(E)-1}\,dr
\end{equation}
for nonnegative Borel $f$, where $\mathrm{tr}(E) = \sum_{j=1}^d \lambda_j$ is the trace of $E$ (equivalently, the homogeneous order of Lebesgue measure on $\R^d$ under the dilation group). We write $\mu_P := \mathrm{tr}(E) > 0$.

The Euclidean case is recovered by taking $E=I$ (so $r^E$ is ordinary scalar dilation), $P(\xi) = |\xi|$ (degree $1$ in this convention), $\mu_P = d$, and $\sigma_P$ equal to the standard surface measure on $S^{d-1}$ with $\sigma_P(S^{d-1}) = \omega_{d-1} = 2\pi^{d/2}/\Gamma(d/2)$.

\subsection*{The homogeneous Gamma identity}

\begin{theorem}[Homogeneous radial Gamma identity]\label{thm:homogeneous}
Under the hypotheses above, the unit sublevel set $B_P := \{\xi\in\R^d : P(\xi)<1\}$ has Lebesgue measure
\begin{equation}\label{eq:homogeneous-gamma}
m(B_P) \;=\; \frac{1}{\Gamma(\mu_P+1)}\int_{\R^d} e^{-P(\xi)}\,d\xi.
\end{equation}
Moreover, $\sigma_P(S_P)\,\Gamma(\mu_P) = \int_{\R^d} e^{-P(\xi)}\,d\xi$, so~\eqref{eq:homogeneous-gamma} is equivalent to $m(B_P) = \sigma_P(S_P)/\mu_P$.
\end{theorem}

\begin{proof}
By~\eqref{eq:bui-randles-polar} applied to $f(\xi) := \Phi(P(\xi))$ for any Borel $\Phi\colon[0,\infty)\to[0,\infty)$, and using $P(r^E\eta) = r$ for $\eta\in S_P$,
\[
\int_{\R^d}\Phi(P(\xi))\,d\xi \;=\; \int_0^\infty\!\int_{S_P}\Phi(r)\,d\sigma_P(\eta)\,r^{\mu_P-1}\,dr \;=\; \sigma_P(S_P)\int_0^\infty \Phi(r)\,r^{\mu_P-1}\,dr.
\]
Take $\Phi(r) = e^{-r}$:
\begin{equation}\label{eq:homog-gauss}
\int_{\R^d} e^{-P(\xi)}\,d\xi \;=\; \sigma_P(S_P)\int_0^\infty e^{-r}\,r^{\mu_P-1}\,dr \;=\; \sigma_P(S_P)\,\Gamma(\mu_P).
\end{equation}
Take $\Phi(r) = \mathbf 1_{(0,1)}(r)$:
\begin{equation}\label{eq:homog-ball}
m(B_P) \;=\; \int_{\R^d}\mathbf 1_{(0,1)}(P(\xi))\,d\xi \;=\; \sigma_P(S_P)\int_0^1 r^{\mu_P-1}\,dr \;=\; \frac{\sigma_P(S_P)}{\mu_P}.
\end{equation}
Eliminating $\sigma_P(S_P)$ between~\eqref{eq:homog-gauss} and~\eqref{eq:homog-ball}, and using $\mu_P\,\Gamma(\mu_P)=\Gamma(\mu_P+1)$,
\[
m(B_P) \;=\; \frac{1}{\mu_P\,\Gamma(\mu_P)}\int_{\R^d}e^{-P(\xi)}\,d\xi \;=\; \frac{1}{\Gamma(\mu_P+1)}\int_{\R^d}e^{-P(\xi)}\,d\xi. \qedhere
\]
\end{proof}

\subsection*{Recovery of the Euclidean formula}

In the rotationally symmetric case $E=I$, $P(\xi) = |\xi|$, $\mu_P = d$, the polar-coordinate computation gives $\int_{\R^d} e^{-|\xi|}\,d\xi = \sigma_P(S_P)\,\Gamma(d)$ via~\eqref{eq:homog-gauss} (note: this is the radial exponential integral against the degree-$1$ norm $|\xi|$, not the squared norm $|\xi|^2$ used elsewhere in the paper). Either way, Theorem~\ref{thm:homogeneous} returns
\[
m(B_d) \;=\; \frac{\sigma_{d-1}(S^{d-1})}{d} \;=\; \frac{2\pi^{d/2}}{d\,\Gamma(d/2)} \;=\; \frac{\pi^{d/2}}{\Gamma(d/2+1)},
\]
which is~\eqref{eq:Vn}. The classical formula is therefore the rotationally symmetric specialization of~\eqref{eq:homogeneous-gamma}.

\subsection*{Scope and comparison with geometric polar frameworks}

Theorem~\ref{thm:homogeneous} is a closed-form identity, not a classification. It takes as input a triple $(E,P,\sigma_P)$ — a contracting dilation generator, a homogeneous defining function, and the attendant Radon surface measure — together with the polar formula~\eqref{eq:bui-randles-polar} that connects them. Given these data, two evaluations of the formula and elimination of $\sigma_P(S_P)$ yield~\eqref{eq:homogeneous-gamma}. The role of the Mellin--Gamma mechanism here is computational: homogeneity supplies the power-law density $r^{\mu_P-1}\,dr$, and Gaussian normalization fixes the surface-mass constant.

This makes precise the relationship between the present paper and the geometric polar framework of Bui--Randles~\cite{BuiRandles2022} and its antecedents. In the Bui--Randles setting, the dimension parameter is fixed (it is the trace $\mu_P=\mathrm{tr}(E)$ of a chosen dilation generator $E$ acting on $\R^d$), and the geometric data $(E,P)$ is primary; the radial measure $r^{\mu_P-1}\,dr$ and the surface measure $\sigma_P$ are derived from it. In the present paper, the situation is inverted: the dimension parameter $x$ is varied, no ambient $\R^d$ or homogeneous function $P$ is assumed, and the radial measure $\mu_x$ is classified directly on $\Cc(\R_{>0})$ from two structural axioms. Theorem~\ref{thm:classification} therefore does not specialize from~\cite{BuiRandles2022}: it produces a measure-valued object on a one-dimensional radial space without any ambient geometry, whereas Bui--Randles construct a measure on $\R^d$ from chosen geometric input. The two frameworks meet only at the level of closed-form identities like~\eqref{eq:homogeneous-gamma}, where the same Mellin--Gamma pattern emerges from both routes. The present paper does not subsume~\cite{BuiRandles2022}, nor is it subsumed by it; the two sit on opposite sides of the constructive/classificatory distinction.

\section{Three routes to the closed form}\label{sec:summary}

The closed form $V(x) = \pi^{x/2}/\Gamma(x/2+1)$ can be reached from any of three logically independent starting points, which we record for comparison.

\smallskip
\noindent\textbf{Interpolation alone.} If one asks for any continuation $V\colon(0,\infty)\to(0,\infty)$ that agrees with $V_n$ at all positive integers, the conditions are underdetermined: a modification by $\sin(\pi z)\,h(z)$ preserves integer values, and Hadamard's 1894 interpolation~\cite{Hadamard1894} provides a concrete second analytic candidate. Integer agreement does not single out the closed form.

\smallskip
\noindent\textbf{Recurrence with log-convexity.} If one imposes the dimension-shift recurrence $V(x+2) = \frac{2\pi}{x+2}V(x)$ with $V(0)=1$, together with log-convexity of an associated function, then by Theorem~\ref{thm:bohrmollerup} the cocycle $T$ — equivalently $V$ — is uniquely determined. This is the Bohr--Mollerup route, and it does not refer to the radial integration operator.

\smallskip
\noindent\textbf{Classification of radial integration.} The route taken in this paper. The radial integration operator $\Iop_x$ is classified directly: scaling covariance forces a power-law density $u^{x/2-1}\,du$, Gaussian normalization fixes its constant to $\pi^{x/2}/\Gamma(x/2)$, and the ball-volume formula is obtained as $\mu_x((0,1))$. The transports $R$ and $T$ emerge as multiplicative dimension-shift cocycles, related by the multiplicative coboundary of $\beta(x)=x$.

\smallskip
The two structural routes (Bohr--Mollerup and the present classification) are independent: neither set of hypotheses entails the other. Both produce the same closed form for $V$. The Gamma function appears in the Bohr--Mollerup route as the unique log-convex generator of the transport recurrence, and in the classification route as the Mellin transform of the Gaussian kernel against the scaling-forced density. The two appearances coincide because the same cocycle $T$ is characterized by both arguments.

The categorical interpretation of Section~\ref{sec:cocycles} is logically separate. Proposition~\ref{prop:cat-coboundary} identifies the coboundary $\beta(x) = x$ with the categorical-dimension functional of the standard object in $\mathrm{Rep}(O_t)$, which provides an external reading of the dimension parameter but does not derive any analytic content from $\mathrm{Rep}(O_t)$. The classification route does not depend on this identification.

\section*{AI use disclosure}
In preparing this manuscript the author made use of several large language models as technical-editorial interlocutors, in particular for prose tightening, structural feedback on the cocycle and categorical-recovery sections, and verification passes on argument flow. All mathematical content---axioms, theorem statements, proofs, and the choice of structural framework---is the author's. The author has independently verified every claim in the paper and accepts full responsibility for its content.

\end{document}